\documentclass[11pt]{article}
\usepackage{amsmath,amssymb,amsthm,mathrsfs}
\usepackage{enumitem}
\usepackage{geometry}
\usepackage{graphicx}

\numberwithin{equation}{section}

\newtheorem{theorem}{Theorem}[section]
\newtheorem{lemma}[theorem]{Lemma}

\newcommand{\vertk}{\stackrel{{\cal D}}{\longrightarrow}}

\title{A 4/7-limit law for the largest interpoint distance in a rotational ellipsoid}
\author{Norbert Henze\footnote{Institute of Stochastics, Karlsruhe Institute of Technology (KIT), 76131 Karlsruhe, Germany. e-mail:
henze@kit.edu} }

\date{}

\begin{document}
\maketitle

\begin{abstract}
Let $M_n$ denote the largest interpoint distance among independent random points
$X_1,\dots,X_n$ uniformly distributed in a compact set in $\mathbb{R}^d$. 
Weak limit laws for $M_n$ are known in several geometric settings, in
particular for ellipsoids with a unique major axis. In this paper we treat
the simplest nontrivial case in which the largest semi-axis is not unique, namely
the rotational ellipsoid $\{(x_1,x_2,x_3)\in\mathbb{R}^3:
(x_1^2+x_2^2)/h^2 + x_3^2/a^2 \le 1\}$,
where $0<a<h$. The diameter of this ellipsoid is attained by all antipodal pairs on the equatorial circle, so the extremal points are not isolated. We prove that
$n^{4/7}(2h-M_n)$ converges in distribution to a Weibull-type limit law with explicit parameter. The proof combines geometric localization arguments with a Chen--Stein Poisson approximation for rare nearly diametral pairs.
\end{abstract}

\medskip


\section{Introduction and statement of the main result}

Let
\[
M_n = \max_{1\le i<j\le n} \|X_i-X_j\|
\]
denote the largest Euclidean interpoint distance among independent and identically
distributed random points $X_1,\ldots,X_n$ in $\mathbb{R}^d$, where $d \ge 2$.

A fundamental problem in stochastic geometry and multivariate extreme-value
theory is to determine the asymptotic distribution of $M_n$ as
$n\to\infty$. Depending on the underlying distribution of the points, the
behavior of $M_n$ may exhibit qualitatively different asymptotic regimes.

If the common distribution has compact support $K$, then
$M_n\to \operatorname{diam}(K)$ almost surely, where
\[
\operatorname{diam}(K)
=
\sup\{\|x-y\|:x,y\in K\}.
\]
In this situation, one seeks normalizing constants $a_n\to\infty$ such that
$a_n\bigl(\operatorname{diam}(K)-M_n\bigr)$
converges in distribution to a nondegenerate limit law.

Weak limit laws for the largest interpoint distance have been obtained in a variety of situations. Matthews and Rukhin \cite{mr93} derived the asymptotic distribution of the largest interpoint distance for Gaussian samples, and Henze and Klein \cite{hekl96} extended these results to symmetric Kotz-type distributions.
 Henze and Lao~\cite{henzelao} investigated the case of spherically
decomposable distributions with Fr\'echet-type heavy tails and showed that,
after suitable normalization, the largest interpoint distance converges to a
nonclassical limit distribution related to a Poisson point process.

For uniformly distributed random points in bounded regions, Appel, Najim and Russo \cite{anr} studied compact planar regions with finitely many major axes under suitable geometric assumptions. Further weak limit laws for bounded supports were obtained by Lao \cite{lao10}, Mayer and Molchanov \cite{maymol07}, Schrempp
\cite{sc16,sc19}, and Heiny and Kleemann \cite{hk25}. Related results for elliptical distributions with unbounded support were obtained by Demichel, Fermin and Soulier \cite{dfs15}. Jammalamadaka and Janson \cite{jamsva15} investigated the maximum interpoint distance for spherically symmetric distributions and discussed several open problems  related to ellipsoids and rotational symmetry.

The present paper considers a problem that naturally emerges from the works of
Schrempp \cite{sc16,sc19} and Jammalamadaka and Janson \cite{jamsva15}, namely the asymptotic behavior of the largest interpoint distance in ellipsoids whose largest semi-axis is not unique. Since the asymptotic behavior of the largest interpoint distance is invariant under multiplication of the underlying point configuration by a positive constant, we may and shall assume without loss of generality that the largest
semi-axis equals one. Thus, in this paper, we consider the rotational ellipsoid in $\mathbb{R}^3$
\begin{equation}\label{ellipsoid}
E = \left\{ (x_1,x_2,x_3)\in\mathbb R^3:
x_1^2+x_2^2 + \frac{x_3^2}{a^2} \le 1 \right\},
\qquad 0<a<1.
\end{equation}

The diameter of $E$ equals $2$. It is attained by all antipodal pairs on the equatorial circle 
\begin{equation}\label{equator}
\mathcal{E} = \left\{ (\cos\theta,\sin\theta,0): 0\le\theta<2\pi \right\}.
\end{equation}
Thus the diameter points are not isolated. For independent and uniformly distributed points $X_1,\ldots,X_n$ on $E$, our main result shows that the correct normalization is $n^{4/7}$, and that the limiting distribution is of Weibull type.

\begin{theorem}\label{thm-intro-main}
Let $0<a<1$. Then
\[
n^{4/7}(2-M_n) \vertk Z,
\]
where $\vertk$ denotes convergence in distribution, and 
\[
\mathbb{P}(Z>t) = \exp\{-\Lambda_{a}t^{7/2}\}, \qquad t\ge 0.
\]
Moreover,
\[
\Lambda_{a} = \frac{9}{64\pi}I_{a},
\]
where
\[
I_{a} = \int_A \mathbf 1 \{G(s,s',y,y',\tau)\le 1\}
\,{\rm d}s\,{\rm d}s'\,{\rm d}y\,{\rm d}y'\,{\rm d}\tau,
\]
\[
G(s,s',y,y',\tau) = \frac{s+s'}{2} + \frac{\tau^2}{4} 
- \frac{a^2}{4}(y-y')^2,
\]
and
\[
A = \{(s,s',y,y',\tau): s,s'\ge 0,\ y^2\le s,\ (y')^2\le s'\}.
\]
\end{theorem}

\smallskip

The proof combines geometric localization arguments with Poisson approximation
for rare extreme pairs. The key geometric feature is that points contributing
to the largest interpoint distance must lie close to the equatorial circle and
must be nearly antipodal. This leads to a nonclassical scaling exponent and to
a Weibull-type limit law.

The paper is organized as follows. In Section~\ref{secmainproof} we prove the main theorem. We introduce coordinates adapted to the equatorial circle, derive a local
expansion of the distance deficit, and prove a localization result for nearly
diametral pairs. The central probabilistic ingredient is a sharp asymptotic
formula for the upper tail of the distance between two independent random
points. A Poisson approximation for rare extreme pairs then yields the limit
distribution in Theorem~\ref{thm-intro-main}. Section~\ref{secproblems} contains further remarks and open problems. In particular, we discuss
the limiting cases $a=0$ and $a=1$ as well as  possible extensions to
higher-dimensional ellipsoids with several equal largest semi-axes.
\medskip

\section{Proof of the main result}\label{secmainproof}

In this section we prove Theorem~\ref{thm-intro-main}. Throughout, we consider
the rotational ellipsoid $E$ given in \eqref{ellipsoid} and independent random points $X_1,X_2,\dots$ uniformly distributed on $E$.

In contrast to the case of a unique major axis, the diameter $2$ of $E$ is attained by all antipodal pairs on the equatorial circle $\mathcal{E}$ defined in \eqref{equator}. 
Thus, if two points have distance close to $2$, they must lie close to
$\mathcal E$, and their angular coordinates must be nearly antipodal.
The asymptotic analysis of the largest interpoint distance therefore reduces to
a local study near the equator. To describe this local geometry, we introduce coordinates adapted to the
equatorial circle.

\subsection{Coordinates}

Every point $x=(x_1,x_2,x_3)\in E$  admits a representation of the form
$x=x(\theta,\delta,w)$, where $0\le\theta<2\pi$, $0\le\delta\le 1$, $|w|\le\sqrt{\delta}$, and 
\[ 
x(\theta,\delta,w) = \Big{(}\sqrt{1-\delta}\cos\theta,\,
\sqrt{1-\delta}\sin\theta,\, aw \Big{)}.
\]
Indeed, $x_1^2+x_2^2 = 1-\delta$ describes the radial defect from the equator in the $(x_1,x_2)$-plane, whereas $w= x_3/a$  describes the displacement in the short direction. The condition $x\in E$ is equivalent to $w^2\le\delta$.

The Jacobian matrix equals
\[
\frac{\partial(x_1,x_2,x_3)}{\partial(\theta,\delta,w)} =
\begin{pmatrix}
-\sqrt{1-\delta}\sin\theta
& -\dfrac{\cos\theta}{2\sqrt{1-\delta}}
& 0
\\[10pt]
\sqrt{1-\delta}\cos\theta
& -\dfrac{\sin\theta}{2\sqrt{1-\delta}}
& 0
\\[10pt]
0 & 0 & a
\end{pmatrix},
\]
and its determinant is $a/2$.

Now, let $X=(X^{(1)},X^{(2)},X^{(3)})$  be uniformly distributed on $E$. We define the random variables $(\Theta,\Delta,W)$ through the representation
$X=x(\Theta,\Delta,W)$. Since $\operatorname{Vol}(E) = 4 \pi a/3$, the change-of-variables formula shows that $(\Theta,\Delta,W)$ has constant density
\[
f_{\Theta,\Delta,W}(\theta,\delta,w)
= \frac{3}{8\pi}
\]
on the parameter domain $0\le\theta<2\pi$, $0\le\delta\le1$, $|w|\le\sqrt{\delta}$.
Thus the random vector $(\Theta,\Delta,W)$ is uniformly distributed on its parameter domain.

\subsection{Local geometry}

We now investigate the geometry of pairs of points whose distance is close to
the diameter $2$. If two points are nearly antipodal and close to the equatorial circle, then their angular coordinates differ approximately by $\pi$. We therefore write the second angular coordinate in the form $\theta+\pi+\psi$, where $\psi$ measures the deviation from exact antipodality.

The next lemma gives a second-order expansion of the distance deficit
$2-\|x-y\|$ in terms of the small parameters $\delta$, $\delta'$, $\psi$, $w$, and $w'$. The following expansion is the key geometric ingredient in the asymptotic analysis of the largest interpoint distance.

\medskip

\begin{lemma}[Local expansion]\label{lem-local-expansion}

Let $x=x(\theta,\delta,w)$ and $y=x(\theta+\pi+\psi,\delta',w')$. Then, as $\delta,\delta',\psi,w,w'\to 0$, we have
\[
2-\|x-y\| = \frac{\delta+\delta'}{2} + \frac{\psi^2}{4}
- \frac{a^2}{4}(w-w')^2 + r(\delta,\delta',\psi,w,w'),
\] 
where
\[
r(\delta,\delta',\psi,w,w') =
o\bigl(\delta+\delta'+\psi^2+w^2+(w')^2\bigr).
\]

The remainder is uniform in $\theta\in[0,2\pi]$.
\end{lemma}

\smallskip
\begin{proof}
Put $r_1=\sqrt{1-\delta}$ and $r_2=\sqrt{1-\delta'}$. Using
$\cos(\pi+\psi) = -1+ \psi^2/2 +o(\psi^2)$ and the trigonometric formula $\cos(s-t) =\cos s \cos t + \sin s \sin t$, we obtain
\begin{align*}
\|x-y\|^2  & = r_1^2+r_2^2-2r_1r_2\cos(\pi+\psi) 
+ a^2(w-w')^2 \\
& = (r_1+r_2)^2 - r_1r_2\psi^2 + a^2(w-w')^2 + o(\psi^2).
\end{align*}

Moreover,
\[
r_1 = 1-\frac{\delta}{2}+o(\delta), \qquad r_2
= 1-\frac{\delta'}{2}+o(\delta').
\]
and therefore $(r_1+r_2)^2 = 4-2(\delta+\delta') + o(\delta+\delta')$.
Also, $r_1r_2 = 1+o(1)$. Hence
\[
\|x-y\|^2 = 4 - 2(\delta+\delta') - \psi^2 + a^2(w-w')^2
+ o(\delta+\delta'+\psi^2+w^2+(w')^2).
\]

Write $\|x-y\| = 2-\eta$. Since $\eta\to 0$, it follows that
\[
(2-\eta)^2  = 4- 4\eta+\eta^2 = 4-4\eta+o(\eta)
\]
and therefore
\[
4\eta = 2(\delta+\delta') + \psi^2 -  a^2(w-w')^2
+ o(\delta+\delta'+\psi^2+w^2+(w')^2).
\]
This proves the assertion.
\end{proof}

\smallskip

Lemma~\ref{lem-local-expansion} shows that, for nearly antipodal points close
to the equatorial circle, the distance deficit $2-\|x-y\|$ admits a quadratic approximation in the variables $\delta,\delta',\psi,w,w'$. The next lemma shows that these variables are indeed small whenever $2-\|x-y\|$ is small. In other words, pairs whose distance is close to the diameter must necessarily lie in a thin neighborhood of the equator and must be nearly antipodal in their angular coordinates. This localization result justifies the use of the local expansion from
Lemma~\ref{lem-local-expansion} in the subsequent asymptotic analysis.

\medskip
\begin{lemma}[Localization]\label{lem-localization}

There exist constants $C>0$ and $\varepsilon_0 > 0$ such that, for
$0<\varepsilon<\varepsilon_0$,
\[
2-\|x(\theta,\delta,w)-x(\theta',\delta',w')\| \le\varepsilon
\]
implies, after writing 
\[
\theta'=\theta+\pi+\psi, \qquad \psi\in[-\pi,\pi],
\]
that
\[
\delta+\delta' \le C\varepsilon,
\]
\[
\psi^2 \le C\varepsilon,
\]
and
\[
w^2+(w')^2 \le C\varepsilon.
\]
\end{lemma}

\smallskip
\begin{proof}
Since $|w|\le\sqrt{\delta}$ and $|w'|\le\sqrt{\delta'}$, we have $(w-w')^2 \le (\sqrt{\delta}+\sqrt{\delta'})^2$.

For fixed $\delta,\delta'$, the largest possible distance occurs for opposite
angular directions and opposite vertical signs. Hence, writing $y=x(\theta', \delta',w')$, 
\[
\|x-y\|^2 \le \big(\sqrt{1-\delta}+\sqrt{1-\delta'}\big)^2 +
a^2 \big(\sqrt{\delta}+\sqrt{\delta'}\big)^2.
\]
As $\delta,\delta'\downarrow 0$,
\begin{align*}
\big(\sqrt{1-\delta}+\sqrt{1-\delta'} \big)^2
& = 4-2(\delta+\delta') + o(\delta+\delta'), \\
\big(\sqrt{\delta}+\sqrt{\delta'}\big)^2
& = \delta+\delta' + 2\sqrt{\delta\delta'}.
\end{align*}

Hence
\begin{align*}
\|x-y\|^2
& \le 4 - (2-a^2)(\delta+\delta')
+ 2a^2\sqrt{\delta\delta'}
+ o(\delta+\delta').
\end{align*}

Since $2\sqrt{\delta\delta'} \le \delta+\delta'$, we obtain $\|x-y\|^2 \le 4-\gamma(\delta+\delta')$  for some $\gamma>0$, because $a<1$.

If $2-\|x-y\| \le \varepsilon$, then $\|x-y\|^2 \ge (2-\varepsilon)^2 = 4-4\varepsilon+O(\varepsilon^2)$. Combining the last two inequalities yields
$\delta+\delta' \le C\varepsilon$. Consequently, $w^2+(w')^2 \le \delta+\delta' \le C\varepsilon$.

Finally, once $\delta+\delta' = O(\varepsilon)$, the planar part of the distance implies $1+\cos(\theta'-\theta) = O(\varepsilon)$. Writing
$\theta'=\theta+\pi+\psi$, this becomes $1-\cos\psi = O(\varepsilon)$.
Since $1-\cos\psi \sim \psi^2/2$, we conclude that $\psi^2 \le C\varepsilon$.
\end{proof}

\medskip

The localization result reduces the asymptotic analysis of the largest interpoint distance to the study of nearly antipodal pairs close to the equatorial circle. The next step is therefore to determine the asymptotic behavior of the probability
$\mathbb{P}(2-\|X_1-X_2\|\le\varepsilon)$  as $\varepsilon\downarrow 0$.

By Lemmas~\ref{lem-local-expansion} and~\ref{lem-localization}, only a small
neighborhood of the equatorial circle contributes asymptotically to this
probability. The following lemma gives the corresponding sharp asymptotic
formula.

\medskip

\subsection{The two-point tail}

For independent uniformly distributed random points $X_1,X_2$ in $E$, put
\[
W_{12}=2-\|X_1-X_2\|.
\]
Recall the quantities $G$, $A$, and $I_{a}$ introduced in Theorem~\ref{thm-intro-main}. 

\medskip

\begin{lemma}[Sharp two-point tail]\label{lem-two-point-tail}
As $\varepsilon\downarrow 0$,
\[
\mathbb{P}(W_{12}\le\varepsilon) \sim K_{a}\varepsilon^{7/2},
\]
where
\[
K_{a} = \frac{9}{32\pi}I_{a}.
\]
\end{lemma}

\smallskip
\begin{proof}
Let $X_1=x(\Theta,\Delta,W)$ and $X_2=x(\Theta',\Delta',W')$ be independent copies. For realizations write $x_1=x(\theta,\delta,w)$ and $x_2=x(\theta',\delta',w')$.
For every $\theta,\theta'\in[0,2\pi)$, there is a unique $\psi\in[-\pi,\pi)$ such that
$\theta'=\theta+\pi+\psi$ modulo $2\pi$.
Thus the transformation
$(\theta,\theta') \longmapsto (\theta,\psi)$
is simply a translation in the second coordinate. Its Jacobian determinant
equals \(1\), and therefore
${\rm d}\theta\,{\rm d}\theta' = {\rm d}\theta\,{\rm d}\psi$.
Hence
\[
\mathbb{P}(W_{12}\le\varepsilon)
= \left(\frac{3}{8\pi}\right)^2
\int_0^{2\pi} \int_{-\pi}^{\pi} \int_0^1 \int_0^1 \int_{-\sqrt{\delta}}^{\sqrt{\delta}} \int_{-\sqrt{\delta'}}^{\sqrt{\delta'}}
\mathbf 1\{D_\varepsilon\}
\,{\rm d}w'\,{\rm d}w\,{\rm d}\delta'\,{\rm d}\delta\,{\rm d}\psi\,{\rm d}\theta,
\]
where
\[
D_\varepsilon = \left\{ 2- \|x(\theta,\delta,w) -
x(\theta+\pi+\psi,\delta',w')\| \le\varepsilon \right\}.
\]

By Lemma~\ref{lem-localization}, there are constants $C>0$ and
$\varepsilon_0>0$ such that, for $0<\varepsilon<\varepsilon_0$, the event
$D_\varepsilon$ implies $\delta+\delta'\le C\varepsilon$, $\psi^2\le C\varepsilon$, and $w^2+(w')^2\le C\varepsilon$.

Now put $\delta=\varepsilon s$, $\delta'=\varepsilon s'$, $w=\sqrt{\varepsilon}\,y$, $w'=\sqrt{\varepsilon}\,y'$, and  $\psi=\sqrt{\varepsilon}\,\tau $.
The Jacobian of this transformation is $\varepsilon^{7/2}$.
Moreover, $|w|\le\sqrt{\delta}$ if and only if  $y^2\le s$. Similarly, $|w'|\le\sqrt{\delta'}$ if and only if  $(y')^2\le s'$.

For $0<\varepsilon<\varepsilon_0$, the preceding localization estimate implies
that the indicator of $D_\varepsilon$, after this change of variables, vanishes
outside a fixed bounded set $A_R = A\cap \{s+s'+y^2+(y')^2+\tau^2\le R\}$,
where $R<\infty$ does not depend on $\varepsilon$.

For $(s,s',y,y',\tau)\in A$, define
\[ H_\varepsilon(\theta,s,s',y,y',\tau)
= \frac{2-\|x(\theta,\varepsilon s,\sqrt{\varepsilon}y)
-  x(\theta+\pi+\sqrt{\varepsilon}\tau, \varepsilon s', 
\sqrt{\varepsilon}y')\|}{\varepsilon}.
\]
By Lemma~\ref{lem-local-expansion}, we have
$H_\varepsilon(\theta,s,s',y,y',\tau) \longrightarrow G(s,s',y,y',\tau)$, uniformly for $(s,s',y,y',\tau)$
in bounded subsets of \(A\), uniformly in \(\theta\in[0,2\pi]\), where
\[
G(s,s',y,y',\tau) = \frac{s+s'}{2} + \frac{\tau^2}{4} - \frac{a^2}{4}(y-y')^2.
\]
Since the indicator function $\mathbf{1}\{u\le1\}$ is discontinuous at the point
$u=1$, convergence of the indicators needs only be verified away from the
set where $G(s,s',y,y',\tau)=1$.
Thus, for every point of $A_R$ satisfying $G(s,s',y,y',\tau)\neq 1$, it follows that
\[
\mathbf 1\{H_\varepsilon(\theta,s,s',y,y',\tau)\le 1\} \longrightarrow
\mathbf 1\{G(s,s',y,y',\tau)\le 1\}.
\]
The exceptional set $\{(s,s',y,y',\tau)\in A_R: G(s,s',y,y',\tau)=1\}$
has five-dimensional Lebesgue measure zero, since $G-1$ is a nonzero polynomial.
Therefore the convergence of indicators holds almost everywhere on $A_R$.
Since the transformed indicators vanish outside the bounded set $A_R$,
they are dominated by the integrable function $\mathbf{1}_{A_R}$. Hence dominated convergence gives
\[
\varepsilon^{-7/2}\mathbb{P}(W_{12}\le\varepsilon)
\longrightarrow \left(\frac{3}{8\pi}\right)^2 \int_0^{2\pi}
\int_A \mathbf 1\{G(s,s',y,y',\tau)\le 1\} \,{\rm d}s\,{\rm d}s'\,{\rm d}y\,{\rm d}y'\,{\rm d}\tau\,{\rm d}\theta.
\]

Since the limiting integrand is independent of $\theta$, this becomes
\[
\left(\frac{3}{8\pi}\right)^2 2\pi I_{a}.
\]
Finally,
\[
\left(\frac{3}{8\pi}\right)^2 2\pi = \frac{9}{32\pi}.
\]
Consequently,
\[
\mathbb{P}(W_{12}\le\varepsilon) \sim \frac{9}{32\pi}I_{a}\varepsilon^{7/2},
\]
as asserted.
\end{proof}

The asymptotic formula in Lemma~\ref{lem-two-point-tail} determines the order
of the probability that a single pair of points is nearly diametral. To obtain a
Poisson limit theorem for the number of such extreme pairs, it remains to control
the dependence between overlapping pairs.


\subsection{Overlap estimate}

The next lemma shows that the probability that two pairs sharing one common
point are simultaneously nearly diametral is of smaller order than the square of
the two-point tail probability. This estimate is the key input for the subsequent
Chen--Stein Poisson approximation.

\begin{lemma}[Overlap estimate]\label{lem-overlap}
As $\varepsilon\downarrow 0$,
\[
\mathbb{P}(W_{12}\le\varepsilon,\ W_{13}\le\varepsilon)
= O(\varepsilon^{11/2}).
\]
\end{lemma}

\smallskip
\begin{proof}
Let $X_1=x(\Theta,\Delta,W)$ be a generic uniformly distributed point in $E$. By Lemma~\ref{lem-localization},  if $W_{12}\le\varepsilon$,
then necessarily $\Delta\le C\varepsilon$ for some constant $C>0$. Hence
\[
\mathbb{P}(W_{12}\le\varepsilon,\ W_{13}\le\varepsilon)
\le \mathbb{E}\left[ \mathbf 1\{\Delta\le C\varepsilon\}
p_\varepsilon(X_1)^2 \right],
\]
where $p_\varepsilon(x) = \mathbb{P}(2-\|x-X_2\|\le\varepsilon)$.

We first estimate $p_\varepsilon(x)$ uniformly for points $x$ satisfying
$\Delta(x)\le C\varepsilon$. Fix such a point $x=x(\theta,\delta,w)$. If
$2-\|x-x(\theta',\delta',w')\|\le\varepsilon$, then Lemma~\ref{lem-localization} implies $\delta'\le C\varepsilon$, $(w')^2\le C\varepsilon$, and $|\psi|\le C\sqrt{\varepsilon}$, where 
$\theta'=\theta+\pi+\psi$ modulo $2\pi$. Thus the admissible set of parameters $(\theta',\delta',w')$ is contained in a
window of size
\[
O(\sqrt{\varepsilon})\times O(\varepsilon)\times O(\sqrt{\varepsilon}).
\]
Since the density of $(\Theta',\Delta',W')$ is constant and equal to
$3/(8\pi)$, it follows that, uniformly in such $x$,
$p_\varepsilon(x) = O(\varepsilon^2)$.
Consequently,
\[
\mathbb{P}(W_{12}\le\varepsilon,\ W_{13}\le\varepsilon)
\le O(\varepsilon^4)
\mathbb{P}(\Delta\le C\varepsilon).
\]
Finally,
\begin{align*}
P(\Delta\le C\varepsilon)
&= \frac{3}{8\pi} \int_0^{2\pi} \int_0^{C\varepsilon}
\int_{-\sqrt{\delta}}^{\sqrt{\delta}} {\rm d}w\,{\rm d}\delta\,{\rm d}\theta
\\
&= \frac{3}{4} \int_0^{C\varepsilon} \sqrt{\delta}\,{\rm d}\delta
\\
&= O(\varepsilon^{3/2}).
\end{align*}

Therefore
$\mathbb{P}(W_{12}\le\varepsilon,\ W_{13}\le\varepsilon)
= O(\varepsilon^4)\,O(\varepsilon^{3/2})
= O(\varepsilon^{11/2})$,
as claimed.
\end{proof} 

\medskip

Lemma~\ref{lem-overlap} shows that the probability of simultaneous occurrence
of two overlapping nearly diametral pairs is asymptotically negligible compared
to the square of the two-point tail probability from
Lemma~\ref{lem-two-point-tail}.

\subsection{Poisson approximation}

Together, Lemmas~\ref{lem-two-point-tail} and~\ref{lem-overlap} provide the
key ingredients for a Poisson approximation of the number of nearly diametral
pairs. We now use a Chen--Stein argument to derive the corresponding Poisson
limit theorem. To this end, put $\varepsilon_n=t n^{-4/7}$ for $t \ge 0$, 
and define 
\[
N_n(t) = \sum_{1\le i<j\le n} I_{ij}^{(n)},
\] 
where $I_{ij}^{(n)} = \mathbf 1\{W_{ij}\le \varepsilon_n\}$ and $W_{ij}=2-\|X_i-X_j\|$.

Thus $N_n(t)$ counts the number of pairs whose distance is within
$\varepsilon_n$ of the diameter $2$. Put
\[
\Lambda_{a} =  \frac12K_{a} = \frac{9}{64\pi}I_{a},
\]
and write Po$(\lambda)$ for the Poisson distribution with expectation $\lambda$.

\medskip

\begin{theorem}[Poisson approximation]\label{thm-poisson}
For every $t\ge 0$,
\[
N_n(t)
\vertk
\operatorname{Po}\bigl(\Lambda_{a}t^{7/2}\bigr).
\]
\end{theorem}


\smallskip
\begin{proof}
If $t=0$, then $\varepsilon_n=0$ and $N_n(0)=0$ almost surely, so the
assertion is trivial. Assume henceforth that $t>0$.
By Lemma~\ref{lem-two-point-tail},
\[
\mathbb{E} N_n(t)
= \binom{n}{2} \mathbb{P}(W_{12}\le\varepsilon_n) \longrightarrow
\frac12 K_{a}t^{7/2}
=
\Lambda_{a}t^{7/2}.
\]

We now prove that $N_n(t)$ is asymptotically Poisson. Let
$\mathcal I_n=\{\{i,j\}:1\le i<j\le n\}$.
For $\alpha=\{i,j\}\in\mathcal I_n$, put
$I_\alpha = I_{ij}^{(n)}$ and $p_\alpha = \mathbb E(I_\alpha)
= \mathbb P(W_{12}\le\varepsilon_n)$.

The family $\{I_{ij}^{(n)}:1\le i<j\le n\}$ has a natural dependency graph:
two indicators are adjacent if and only if the corresponding unordered index
pairs have a nonempty intersection. Accordingly, for
$\alpha\in\mathcal I_n$, define
$B_\alpha = \{\beta\in\mathcal I_n:\alpha\cap\beta\ne\varnothing\}$.
Then $|B_\alpha|\le 2n$, and $I_\alpha$ is independent of the family
$\{I_\beta:\beta\notin B_\alpha\}$.

We apply Theorem~1 in Arratia, Goldstein and Gordon \cite{agg}. Since
$I_\alpha$ is independent of $\{I_\beta:\beta\notin B_\alpha\}$,
the term $b_3$ in that theorem is equal to zero. Moreover,
$b_1 = \sum_{\alpha\in\mathcal I_n} \sum_{\beta\in B_\alpha} p_\alpha p_\beta$.

Since $|\mathcal I_n|=\binom n2$, $|B_\alpha|\le 2n$, and
$p_\alpha=\mathbb P(W_{12}\le\varepsilon_n)$, we obtain
$b_1 \le C n^3 \mathbb P(W_{12}\le\varepsilon_n)^2$.
Further,
\[
b_2 = \sum_{\alpha\in\mathcal I_n}
\sum_{\substack{\beta\in B_\alpha\\ \beta\ne\alpha}}
\mathbb E(I_\alpha I_\beta).
\]
If $\alpha\ne\beta$ and $\beta\in B_\alpha$, then the two corresponding pairs
share exactly one index. By exchangeability,
\[
\mathbb E(I_\alpha I_\beta) =
\mathbb P(W_{12}\le\varepsilon_n,\ W_{13}\le\varepsilon_n)
\]
and therefore
$b_2 \le  C n^3 \mathbb P(W_{12}\le\varepsilon_n,\ W_{13}\le\varepsilon_n)$.
Consequently, Theorem~1 in \cite{agg} yields
\[
 d_{\mathrm{TV}} \left(\mathcal L(N_n(t)),
\operatorname{Po}(\mathbb E N_n(t))
\right) \le C \left[ n^3 \mathbb P(W_{12}\le\varepsilon_n)^2
+ n^3 \mathbb P(W_{12}\le\varepsilon_n,\ W_{13}\le\varepsilon_n)
\right],
\]
where $C$ is an absolute constant, and $\mathcal{L}(N_n(t))$ denotes the
distribution of $N_n(t)$.

It remains to show that the expression on the right-hand side tends to zero.
By Lemma~\ref{lem-two-point-tail}, $\mathbb P(W_{12}\le\varepsilon_n)
= O(\varepsilon_n^{7/2})$ and thus
\[
n^3 \mathbb P(W_{12}\le\varepsilon_n)^2
= O(n^3\varepsilon_n^7)
= O(n^3 n^{-4})
= O(n^{-1})
\to 0.
\]

Furthermore, by Lemma~\ref{lem-overlap},
$\mathbb P(W_{12}\le\varepsilon_n,\ W_{13}\le\varepsilon_n)
= O(\varepsilon_n^{11/2})$.
Therefore
\[
n^3 \mathbb P(W_{12}\le\varepsilon_n,\ W_{13}\le\varepsilon_n)
= O(n^3\varepsilon_n^{11/2})
= O(n^3 n^{-22/7})
= O(n^{-1/7})
\to 0.
\]

Consequently,
$ d_{\mathrm{TV}} \left( \mathcal L(N_n(t)), \operatorname{Po}(\mathbb E N_n(t))
\right) \to 0$.
Since $\mathbb{E} N_n(t)\to \Lambda_{a}t^{7/2}$,
we obtain
\[ N_n(t) \vertk \operatorname{Po}(\Lambda_{a}t^{7/2}),
\]
which proves the theorem.
\end{proof}

\medskip

\subsection{Completion of the proof of Theorem~\ref{thm-intro-main}}

We now complete the proof of Theorem~\ref{thm-intro-main}. Fix $t\ge 0$. By the
definition of $N_n(t)$,
\[
\{n^{4/7}(2-M_n)>t\}
=
\{N_n(t)=0\}.
\]
Indeed, the left-hand event says that every interpoint distance is smaller
than $2-tn^{-4/7}$, which is equivalent to saying that no pair satisfies
$W_{ij}\le tn^{-4/7}$.

By Theorem~\ref{thm-poisson},
$N_n(t)\vertk \operatorname{Po}(\Lambda_a t^{7/2})$ and thus 
$
\mathbb{P}(N_n(t)=0)
\to  \exp\{-\Lambda_a t^{7/2}\}$.
Consequently,
\[
\mathbb{P}\{n^{4/7}(2-M_n)>t\}
\to
\exp\{-\Lambda_a t^{7/2}\},
\]
which proves Theorem~\ref{thm-intro-main}.

\section{Remarks and open problems}\label{secproblems}

This section discusses several aspects related to Theorem~\ref{thm-intro-main}.
We first compare the present result with the limiting cases $a=0$ and
$a=1$, where weak limit laws are known from the work of Mayer and
Molchanov~\cite{maymol07}. We then formulate an open problem for
higher-dimensional ellipsoids with non-unique largest semi-axes. 

\medskip

\subsection{The limiting cases $a=0$ and $a=1$}
The behaviour in Theorem~\ref{thm-intro-main} exhibits an interesting discontinuity at
the limiting cases $a=0$ and $a=1$.

If $a=0$, the support degenerates to the unit circle in $\mathbb{R}^2$. In
this case, the largest interpoint distance $M_n$ is governed by the geometry of the
boundary of the unit disk, and Mayer and Molchanov \cite[Theorem~1.1]{maymol07}
showed that in this case $n^{4/5}(2-M_n)$ has a limiting Weibull distribution.

If $a=1$, the support becomes the three-dimensional unit ball. Again by
Theorem~1.1 in \cite{maymol07}, the correct normalization is
$n^{2/3}$.

By contrast, Theorem~\ref{thm-intro-main} shows that for every fixed
$a\in(0,1)$, the correct normalization is $n^{4/7}$.

Thus the asymptotic behavior changes discontinuously at both limiting cases
$a=0$ and $a=1$. In particular, the exponent $4/7$ does not arise as a
continuous interpolation between the exponents $4/5$ and $2/3$.
This phenomenon reflects the fact that the geometry of the set of diameter
points changes qualitatively in the limiting regimes.

\medskip

\subsection{Higher-dimensional ellipsoids}
Let $d\ge 4$, and let
\[
E
=
\left\{
x=(x_1,\ldots,x_d)\in\mathbb{R}^d:
\frac{x_1^2}{h_1^2}
+\cdots+
\frac{x_d^2}{h_d^2}
\le1
\right\},
\]
where
\[
1=h_1\ge h_2\ge\cdots\ge h_d>0.
\]
Let $X_1, \ldots ,X_n$ be independent and uniformly distributed on $E$.

If $h_2<1$, then the ellipsoid has a unique major axis, and the asymptotic behavior of
$2-M_n$ is known; see \cite[Corollary~2.1]{sc16}. In this case, the diameter
is attained only at the two antipodal endpoints of the major axis.
The present paper treats the simplest case in which the largest semi-axis is not
unique, namely the rotational ellipsoid in $\mathbb{R}^3$. A natural open
problem is to determine the asymptotic behavior of $2-M_n$ in higher-dimensional
ellipsoids for which $h_j=1$ for some
$j\in\{3,\ldots,d-1\}$ and $h_d<h_{d-1}$.
In this situation, the set of diameter points is no longer finite. Rather, it
forms a sphere of positive dimension. More precisely, if
\[
h_1=\cdots=h_k=1>h_{k+1}\ge\cdots\ge h_d,
\]
then the diameter is attained by all antipodal pairs on
$\mathbb{S}^{k-1}\times\{0\}^{d-k}$.
Thus the geometry of extremal pairs changes qualitatively compared with the
case of a unique major axis.

The results of the present paper suggest that the asymptotic behavior of
$2-M_n$ should again be governed by a nonclassical Weibull-type limit law,
whose normalization depends on the dimension of the manifold of diameter points.
Determining the correct scaling exponent and the corresponding limit
distribution appears to be an interesting open problem.

\end{document}